\documentclass[11pt,reqno]{amsart}
\usepackage[utf8]{inputenc}
\usepackage[T1]{fontenc}
\usepackage{lmodern}
\usepackage{amssymb,mathrsfs}
\usepackage[margin=1in,headsep=0.25in]{geometry}
\usepackage[colorlinks,linkcolor=blue,citecolor=red,urlcolor=blue]{hyperref}

\theoremstyle{plain}
\newtheorem{theorem}{Theorem}[section]
\newtheorem{lemma}[theorem]{Lemma}
\newtheorem{proposition}[theorem]{Proposition}

\theoremstyle{remark}
\newtheorem{remark}[theorem]{Remark}
\numberwithin{equation}{section}

\newcommand{\E}{\mathbb E}
\newcommand{\R}{\mathbb R}

\newcommand{\norm}[1]{\left\lVert #1\right\rVert}

\newcommand{\dd}{\,\mathrm d}
\newcommand{\e}{\mathrm e}

\author[Shuo Qin]{Shuo Qin}
\address[Shuo Qin]{Beijing Institute of Mathematical Sciences and Applications, and Yau Mathematical Sciences Center, Tsinghua University}
\email{qinshuo@bimsa.cn}

\title{Homogenization of weakly reinforced\\P\'olya urns on countable networks}
\date{}


\begin{document}
\begin{abstract}
We consider weakly reinforced P\'olya urns on a countable graph of uniformly bounded degree. The vertices have positive bounded firing rates, which need not be bounded away from zero. For every reinforcement exponent $\alpha\in[0,1)$ and every deterministic choice of finite positive initial weights, the normalized edge weights converge almost surely, coordinate-wise, to a deterministic vector independent of the initial weights. We identify this vector as the unique non-vanishing equilibrium of the limiting equation. This proves the homogenization conjecture of Couzini\'e and Hirsch, including its formulation for non-uniform firing rates in the MATRIX open-problems report. The main step is a Liouville theorem for the limiting equation, established by a logarithmic-norm comparison of complete solutions using the row-sum balance at each vertex.
\end{abstract}
\maketitle

\section{Introduction}
\label{sec:introduction}

\subsection{Definition and main result}

In reinforced random processes, outcomes that have occurred more often receive greater weight in subsequent choices. A paradigmatic example is P\'olya's urn, in which each drawn ball is returned to the urn together with another ball of the same color. We refer to \cite{MR2282181} for a survey of reinforced processes and their applications.

In this paper, we study the graph version of the $(W,A)$-reinforcement model (WARM), introduced by van der Hofstad, Holmes, Kuznetsov and Ruszel \cite{MR3543903} as a toy model for the formation of neuronal connections. The colors are the edges of a graph. At each step, the edges incident to a randomly selected vertex compete for reinforcement, with probabilities proportional to their current weights raised to a power $\alpha$. The authors studied the stability of equilibria and phase transitions in the strong reinforcement regime $\alpha>1$.

Couzini\'e and Hirsch \cite{MR4275961} studied the model on countable networks, in continuous time and with unit initial weights and unit firing rates. Here we allow general positive initial weights and bounded firing rates. Let $G=(V,E)$ be a countable connected simple graph such that
\begin{equation}\label{eq:degree-bound}
    1\leq\deg(v)\leq\Delta<\infty
    \qquad\text{for every }v\in V.
\end{equation}
Equip each vertex $v$ with an independent Poisson clock of rate $p_v$, called its \emph{firing rate}, and assume that
\begin{equation}\label{eq:rate-bound}
    0<p_v\leq p_*<\infty
    \qquad\text{for every }v\in V.
\end{equation}
For $v\in V$, write $E_v:=\{e\in E:v\in e\}$. The initial weights $N_0(e)$ are deterministic and satisfy $0<N_0(e)<\infty$ for every edge. No uniform bounds on these initial weights are imposed. Attach independent uniform marks on $[0,1]$ to all clock rings, independently of the clocks, and fix an ordering of $E_v$ at every vertex. If the clock at $v$ rings at time $t$, use its mark to select $e\in E_v$ with conditional probability
\begin{equation}\label{eq:selection-probability}
    \frac{N_{t-}(e)^\alpha}
    {\sum_{\tilde e\in E_v}N_{t-}(\tilde e)^\alpha},
\end{equation}
by partitioning $[0,1]$ into intervals of these lengths, and increase $N_t(e)$ by one. Throughout the paper, the reinforcement exponent lies in $[0,1)$, the weak reinforcement regime. 

When the graph $G$ is finite and all vertices fire at the same rate, the associated discrete process is the discrete-time WARM process studied in \cite{MR3543903, MR3631695}. The graphical construction above was given by Hirsch, Holmes and Kleptsyn, and they showed in \cite[Section~2]{MR4254478} that the process is also well-defined in the infinite graph setting: the weights are built from the clock rings and their selection marks, and an exploration argument excludes infinite chains of dependencies among the clocks.  This exploration uses only the clocks, so it also applies to the arbitrary finite positive initial weights considered here. Note that, although $\alpha >1$ is assumed in \cite{MR4254478}, the case $\alpha \in [0,1]$ is covered by the same proof. 

We write
\[
    X_t(e):=\frac{N_t(e)}t,\qquad t>0,
\]
for the normalized edge weights. For a positive vector $x=(x_e)_{e\in E}$, define
\begin{equation}\label{eq:S-Phi-F}
    S_v(x):=\sum_{f\in E_v}x_f^\alpha,\qquad
    \Phi_e(x):=\sum_{v\in e}\frac{p_vx_e^\alpha}{S_v(x)},\qquad
    F_e(x):=\Phi_e(x)-x_e.
\end{equation}
After a logarithmic time change, the normalized weights satisfy a stochastic-approximation equation with drift $F$, as noted in \cite[Section~5]{Aletti2025}. We derive the exact equation in \eqref{eq:log-SA}. The associated differential equation is
\begin{equation}\label{eq:x-ode}
    \dot x_e=F_e(x)
    =-x_e+\sum_{v\in e}\frac{p_vx_e^\alpha}{S_v(x)},
    \qquad e\in E.
\end{equation}
A vector $\mu\in(0,\infty)^E$ is called a \emph{non-vanishing equilibrium} if $F(\mu)=0$ coordinate-wise. Here non-vanishing means that every coordinate is strictly positive, without requiring a uniform positive lower bound.

Couzini\'e and Hirsch proved homogenization on every graph of uniformly bounded degree when $\alpha<1/2$, extended the range slightly beyond $1/2$ on regular graphs, and treated the whole interval $\alpha<1$ on $\mathbb Z$. They conjectured \cite[Conjecture~2.1]{MR4275961} that homogenization holds on every countable graph of uniformly bounded degree throughout the weak reinforcement regime: there is a unique non-vanishing equilibrium $\mu$, and $X_t(e)\to\mu_e$ almost surely for every edge $e$. A formulation with arbitrary positive bounded firing rates appears in \cite[Section~5, Conjecture~1]{Aletti2025}. Those rates are not assumed to be bounded away from zero. We prove the conjecture in this generality, and allow arbitrary deterministic finite positive initial weights.

\begin{theorem}[Homogenization]\label{thm:main}
Let $G$ and $(p_v)_{v\in V}$ satisfy \eqref{eq:degree-bound} and \eqref{eq:rate-bound}, and let $\alpha\in[0,1)$. There is a unique non-vanishing equilibrium $\mu$ of \eqref{eq:x-ode}. For every deterministic choice of finite positive initial weights, almost surely,
\begin{equation}\label{eq:main-convergence}
    \lim_{t\to\infty}\frac{N_t(e)}t=\mu_e
    \qquad\text{for every }e\in E.
\end{equation}
\end{theorem}

\begin{remark}
After the present work was completed, the author learned that Angel, Hermon and Yao \cite{YaoAngelHermon} independently proved the homogenization conjecture of \cite{MR4275961} by different arguments.
\end{remark}

The equilibrium depends only on $G$, the firing rates and $\alpha$. Thus homogenization here means convergence to a deterministic equilibrium profile independent of the initial weights. Its coordinates need not be equal. Explicit upper and lower bounds on $\mu$ in terms of the firing rates are given in Proposition \ref{prop:equilibrium-existence}. If every vertex has degree $d$ and $p_v\equiv p$, the constant vector $\mu_e=2p/d$ satisfies the equilibrium equation. Theorem \ref{thm:main} therefore identifies the limit explicitly in this case.

\begin{remark}[Boundary equilibria]\label{rem:vanishing-equilibria}
Uniqueness in Theorem \ref{thm:main} is asserted among non-vanishing equilibria. When $0<\alpha<1$, the equation also makes sense for nonnegative weights whenever $S_v(x)>0$ for every vertex $v$. For example, with unit firing rates on $\mathbb Z$, assigning weight $2$ to the edges of a perfect matching and weight $0$ to the remaining edges gives an equilibrium with zero coordinates. Proposition \ref{prop:persistence} excludes such boundary equilibria as limits of the urn process.
\end{remark}

\begin{remark}[The case $\alpha=0$]\label{rem:alpha-zero}
When $\alpha=0$, each incident edge is selected with probability $1/\deg(v)$ at vertex $v$. Poisson thinning and the strong law give, for $e=\{u,v\}$,
\[
    X_t(e)\longrightarrow
    \frac{p_u}{\deg(u)}+\frac{p_v}{\deg(v)}
    \qquad\text{almost surely}.
\]
This vector is also the unique equilibrium. Henceforth we assume $0<\alpha<1$.
\end{remark}

\subsection{Related work}\label{sec:related-work}

Bena\"im, Benjamini, Chen and Lima \cite{MR3346459} studied a related model with weights attached to vertices. Each edge allocates a new ball to one of its endpoints according to their reinforced weights. On finite connected graphs, they proved that the normalized vertex weights converge almost surely to a deterministic vector with strictly positive coordinates for $0<\alpha<1$. The linear case was treated by Chen and Lucas \cite{MR3269167}. Angel, Hermon and Yao \cite{YaoAngelHermon} study the model on countable networks where each edge is equipped with an independent Poisson clock, and prove that in the weak reinforcement regime, the normalized vertex weights on trees converge almost surely coordinate-wise to a positive vector. 

Other interacting reinforcement models exhibit different forms of convergence. Aletti, Crimaldi and Ghiglietti \cite{MR3737938} studied interacting urns whose sampling probabilities are weighted averages of the urn proportions across the network. Under suitable assumptions, the proportions converge almost surely to a common limit, a phenomenon called synchronization. Liggett and Rolles \cite{MR2078537} studied an infinite model of social network formation with reciprocal reinforcement of neighbor choices. They characterized its extremal invariant measures by partitions of the population into finite stars.

For WARM in the strong reinforcement regime, Hirsch, Holmes and Kleptsyn \cite{MR4254478} proved that, on graphs of uniformly bounded degree with unit initial weights and positive bounded firing rates, every connected component of the subgraph of edges ever reinforced is almost surely finite when $\alpha$ is sufficiently large. Their subsequent work \cite{MR4580686} studies the components formed by edges reinforced infinitely often and phase transitions in their structure as $\alpha$ varies.

\subsection{Proof strategy and organization of the paper}\label{sec:proof-strategy}

Our proof follows the dynamical approach to stochastic approximation described in \cite{MR1767993}. On finite graphs, the analyses in \cite{MR3346459, MR3543903} use Lyapunov functions for the associated differential equations. On a countable graph, the sums defining those functions need not be finite. We instead prove a Liouville theorem: every bounded positive complete solution of \eqref{eq:x-ode} equals the unique non-vanishing equilibrium. Here, a complete solution is one defined for every real time. We note that the uniqueness of the equilibrium alone would not suffice for our proof of Theorem \ref{thm:main}.

Section \ref{sec:ode} proves this Liouville theorem. Existence of an equilibrium follows from a fixed-point argument. For a positive solution $x$, the change of variables $y_e=x_e^{1-\alpha}$ removes the edge factor from the reinforcement term:
\[
    \dot y_e=(1-\alpha)\left(-y_e+
    \sum_{v\in e}\frac{p_v}{\sum_{f\in E_v}y_f^{\alpha/(1-\alpha)}}\right).
\]
Completeness of the solution gives uniform control of the vertex contributions in this equation, even without uniform positive lower bounds on the firing rates or the coordinates. Comparing two complete solutions then leads to an equation on the vertices, see \eqref{eq:q-equation}. Each edge contributes once to each endpoint, so the diagonal interaction coefficient equals the sum of the off-diagonal coefficients in its row. A supremum-norm estimate based on this balance gives decay of differences over sufficiently long intervals. Applying it on intervals that extend arbitrarily far into the past proves the Liouville theorem.

Section \ref{sec:sa} connects the urn process to complete solutions of the limiting equation. We first prove persistence, meaning that $\liminf_{t\to\infty}X_t(e)>0$ almost surely for every edge. Couzini\'e and Hirsch proved this for unit initial weights and unit firing rates in \cite[Lemma~3.6]{MR4275961}. A comparison with a pure-birth process covers the more general assumptions here. On logarithmic time, the martingale increments in the stochastic-approximation equation \eqref{eq:log-SA} tend to zero uniformly on bounded time intervals. Local upper bounds, persistence and a diagonal Arzel\`a--Ascoli argument then show that every sequence of time shifts tending to infinity has a subsequence converging, coordinate-wise and uniformly on compact time intervals, to a bounded positive complete solution.  The Liouville theorem identifies all these limits with the unique equilibrium and proves convergence.

\section{The limiting differential equation}\label{sec:ode}

By a solution of \eqref{eq:x-ode} on an interval $I$ we mean a map $x:I\to(0,\infty)^E$ whose coordinates are continuously differentiable and satisfy the equation.   No uniform positive lower bound is required.  A solution is \emph{complete} if $I=\R$.  Such solutions arise as limits of time shifts in Section \ref{sec:sa}.  We call a complete solution bounded and positive if
\[
 x_e(t)>0, \ \forall t\in \R, e\in E \ \text{ and }   \sup_{t\in\R,\,e\in E}x_e(t)<\infty.
\]
The main result of this section is the following.

\begin{theorem}[Liouville property]\label{thm:liouville}
There is a unique non-vanishing equilibrium $\mu$ of \eqref{eq:x-ode}.  Every bounded positive complete solution of \eqref{eq:x-ode} is identically equal to $\mu$.
\end{theorem}

We write $\ell^\infty(E)$ for the space of bounded edge vectors, with norm $\norm z_\infty=\sup_{e\in E}|z_e|$.  The same notation is used for vertex vectors.  The notation $C^1(I;\ell^\infty(E))$ means that the difference quotients converge in this norm and that the derivative is norm-continuous.  Coordinate-wise continuous differentiability alone does not imply this property.

We repeatedly use the variation-of-constants formula.  If a scalar function satisfies $\dot z+az=g$ on $[s,t]$, where $a>0$ is constant and $g$ is continuous, then
\begin{equation}\label{eq:variation-of-constants}
    z(t)=\e^{-a(t-s)}z(s)
       +\int_s^t\e^{-a(t-r)}g(r)\dd r.
\end{equation}
We apply the formula coordinate-wise throughout.

We will also use a criterion for differentiability in the supremum norm.  If $z$ takes values in $\ell^\infty(E)$, each coordinate satisfies $\dot z_e=g_e$, and $g$ is continuous in the supremum norm, then $z\in C^1(I;\ell^\infty(E))$.  Indeed,
\begin{equation}\label{eq:sup-diff}
    \norm{\frac{z(t+\varepsilon)-z(t)}{\varepsilon}-g(t)}_\infty
    \leq\sup_{|r-t|\leq|\varepsilon|}\norm{g(r)-g(t)}_\infty
    \longrightarrow 0
\end{equation}
as $\varepsilon\to0$, for $t$ in the interior of $I$.  The same argument applies on the vertex space.

\subsection{Equilibria and elementary bounds}

For $e=\{u,v\} \in E$, put
\begin{equation}\label{eq:edge-bounds}
    b_e:=p_u+p_v,\qquad
    \eta_e:=\frac{p_u}{\deg(u)}+\frac{p_v}{\deg(v)},\qquad
    m_e:=\left(\frac{\eta_e}{(2 p_{*})^\alpha}\right)^{1/(1-\alpha)}.
\end{equation}
Notice that $0<m_e\leq\eta_e\leq b_e\leq2 p_{*}$. In particular, for any $x \in (0,\infty)^E$, one has $\Phi_e(x) \leq b_e \leq2 p_{*}$.

\begin{proposition}\label{prop:equilibrium-existence}
Equation \eqref{eq:x-ode} has a non-vanishing equilibrium.  Every such equilibrium satisfies
\begin{equation}\label{eq:mu-bounds}
    m_e\leq\mu_e\leq b_e
    \qquad\text{for every }e\in E.
\end{equation}
\end{proposition}

\begin{proof}
We adapt the fixed-point argument of \cite[Proposition~3.2]{MR4275961}.  We use the Schauder--Tychonoff theorem: a continuous map from a nonempty compact convex subset of a locally convex topological vector space into itself has a fixed point.  Here the ambient space is the locally convex space $\R^E$ with the product topology, in which convergence means coordinate-wise convergence. The set
\[
    \mathcal K:=\prod_{e\in E}[m_e,b_e]
\]
is nonempty and convex.  It is compact as a countable product of compact intervals, which can also be seen by a diagonal subsequence argument.  If $x\in\mathcal K$, then $\Phi_e(x)\leq b_e$ and $S_v(x)\leq\deg(v)(2 p_{*})^\alpha$.  Consequently, for $e=\{u,v\}$,
\[
    \Phi_e(x)\geq
    \frac{\eta_e x_e^\alpha}{(2 p_{*})^\alpha}
    \geq
    \frac{\eta_e m_e^\alpha}{(2 p_{*})^\alpha}
    =m_e.
\]
Thus $\Phi$ maps $\mathcal K$ into itself.  For each fixed vertex $v$,
\[
    S_v(x)\geq\sum_{f\in E_v}m_f^\alpha>0,
    \qquad x\in\mathcal K.
\]
Since each coordinate of $\Phi$ depends on finitely many coordinates of $x$, these positive denominator bounds give continuity in the product topology.  The fixed-point theorem therefore gives an equilibrium in $\mathcal K$.

Conversely, a non-vanishing equilibrium satisfies $\mu_e=\Phi_e(\mu)\leq b_e\leq2 p_{*}$.  The same denominator bound $S_v(\mu)\leq\deg(v)(2 p_{*})^\alpha$ then gives
\[
  \mu_e=  \Phi_e(\mu) \geq\frac{\eta_e  \mu_e^{\alpha}}{(2 p_{*})^\alpha}, \quad \text{ i.e. }\  \mu_e^{1-\alpha}\geq\frac{\eta_e }{(2 p_{*})^\alpha}
\]
which proves the lower bound.
\end{proof}

\subsection{An estimate for complete trajectories}

Write
\begin{equation}\label{eq:gamma-beta}
    \gamma:=1-\alpha,\qquad \beta:=\frac{\alpha}{\gamma}.
\end{equation}
For a positive solution $x$, set $y_e=x_e^\gamma$.  Then
\begin{equation}\label{eq:y-ode}
    \dot y_e=\gamma\left(-y_e+\sum_{v\in e}\rho_v(y)\right),
    \qquad \text{where } R_v(y):=\sum_{f\in E_v}y_f^\beta,\quad
    \rho_v(y):=\frac{p_v}{R_v(y)}.
\end{equation}
Since $\gamma\beta=\alpha$, we have $R_v(y)=S_v(x)$.  The power $\gamma$ removes the edge factor from the forcing term.  Indeed, for a general power $z_e=x_e^\kappa$,
\[
    \dot z_e=-\kappa z_e
      +\kappa x_e^{\kappa+\alpha-1}
          \sum_{v\in e}\frac{p_v}{S_v(x)}.
\]
The exponent of $x_e$ vanishes when $\kappa=1-\alpha$. 

\begin{lemma}\label{lem:complete-feedback}
Let $x$ be a bounded positive complete solution of \eqref{eq:x-ode}, and let $y_e=x_e^\gamma$.  Then $m_e\leq x_e(t)\leq b_e$ for every $e\in E$ and $t\in\R$, and
\begin{equation}\label{eq:complete-feedback}
    y_e(t)\geq\gamma\sum_{v\in e}\rho_v(y(t))
    \qquad \text{for any } e\in E,\ t\in\R .
\end{equation}
Moreover,
\begin{equation}\label{eq:log-y-bound}
    -\gamma\leq\frac{\dot y_e(t)}{y_e(t)}\leq\alpha,
\end{equation}
and $y$ is continuously differentiable as an $\ell^\infty(E)$-valued map.
\end{lemma}

\begin{proof}
Since $\dot x_e+x_e=\Phi_e(x)$ and $\Phi_e\leq b_e$, formula \eqref{eq:variation-of-constants} gives, for any $T>0$,
\[
    x_e(t)\leq\e^{-T}x_e(t-T)+(1-\e^{-T}) b_e.
\]
Boundedness of $x$ and letting $T\to\infty$ imply $x_e(t)\leq b_e\leq2 p_{*}$. On the other hand, from $\dot x_e\geq-x_e$, we have $(\e^t x_e(t))'\geq0$. Therefore, $x_e(t-\tau)\leq\e^\tau x_e(t)$ for any $\tau\geq 0$.  Since $R_v(y)=\sum_{f\in E_v}x_f^\alpha$, it follows that
\[
    R_v(y(t-\tau))\leq\e^{\alpha\tau}R_v(y(t)),
    \qquad
    \rho_v(y(t-\tau))\geq\e^{-\alpha\tau}\rho_v(y(t)).
\]
Apply \eqref{eq:variation-of-constants} to \eqref{eq:y-ode} on $[t-T,t]$ and change variables to obtain
\[
    y_e(t)=\e^{-\gamma T}y_e(t-T)
      +\gamma\int_0^T\e^{-\gamma\tau}
          \sum_{v\in e}\rho_v(y(t-\tau))\dd\tau.
\]
The completeness of the solution allows us to use this identity for every $T>0$.  Its first term is bounded by $(2 p_{*})^\gamma\e^{-\gamma T}$ and tends to zero as $T\to\infty$.  The integrand is nonnegative, so letting $T\to\infty$ yields \eqref{eq:complete-feedback}:
\begin{align*}
    y_e(t)=\gamma\int_0^\infty\e^{-\gamma\tau}
        \sum_{v\in e}\rho_v(y(t-\tau))\dd\tau \geq\gamma\sum_{v\in e}\rho_v(y(t))
        \int_0^\infty\e^{-(\gamma+\alpha)\tau}\dd\tau
     =\gamma\sum_{v\in e}\rho_v(y(t)),
\end{align*}
where we used $\gamma+\alpha=1$ in the last equality.  The upper bound on $x$ also gives
\[
    \sum_{v\in e}\rho_v(y(r))
    \geq\frac{\eta_e}{(2p_*)^\alpha}=m_e^\gamma
    \qquad\text{for every }r\in\R.
\]
Substituting this bound into the same integral representation gives $y_e(t)\geq m_e^\gamma$, hence $x_e(t)\geq m_e$.  Finally, \eqref{eq:y-ode} and \eqref{eq:complete-feedback} give
\[
 -\gamma y_e \leq  \dot y_e \leq \gamma(-y_e+y_e/\gamma) = \alpha y_e,
\]
which proves \eqref{eq:log-y-bound}.    

It remains to check differentiability in the supremum norm.  Put $L:=\max\{\gamma,\alpha\}$.  By \eqref{eq:log-y-bound}, $|\dot y_e|\leq Ly_e$. Differentiating the finite sum defining $R_v$ then gives
\begin{equation}
    \label{eq:log-R-bound}
      \left|\frac{\dd}{\dd t}R_v(y)\right|
    \leq\sum_{e\in E_v}\beta y_e^{\beta-1}|\dot y_e|
    \leq\beta L R_v(y).
\end{equation}
Since $\rho_v=p_v/R_v$, this implies $|\frac{\dd}{\dd t}\rho_v(y)|\leq\beta L\rho_v(y)$. For each vertex $v$, choose any incident edge in \eqref{eq:complete-feedback} to get $\rho_v(y)\leq(2 p_{*})^\gamma/\gamma$. Thus, the coordinate derivatives of both $y$ and $\rho(y)$ are uniformly bounded.  Integrating these bounds makes both maps Lipschitz in their respective supremum norms.  By \eqref{eq:y-ode}, the vector $\dot y$ is therefore continuous in the supremum norm.  The criterion \eqref{eq:sup-diff} now shows that $y\in C^1(\R;\ell^\infty(E))$.
\end{proof}

\subsection{An estimate for differences}

For a bounded linear operator $A$ on $\ell^\infty(V)$, write
\[
    \norm A_{\infty\to\infty}
    :=\sup_{\norm z_\infty\leq1}\norm{Az}_\infty.
\]
Its logarithmic norm is defined by
\[
    \operatorname{Log}_\infty(A)
    :=\lim_{\varepsilon\to 0+}
      \frac{\norm{\mathrm{Id}+\varepsilon A}_{\infty\to\infty}-1}{\varepsilon},
\]
where $\mathrm{Id}$ is the identity operator.  It measures the first-order change of the norm under $\mathrm{Id}+\varepsilon A$.  The limit exists by convexity of the operator norm as a function of $\varepsilon$.  We use the supremum-norm estimate below, following the logarithmic-norm argument in \cite[Proposition~6.4]{MR4864220}.  Its key assumption is that the additional diagonal coefficient equals the off-diagonal row sum. 

\begin{lemma}[Supremum-norm estimate]\label{lem:lognorm-estimate}
Let $I=[s,t]$, let $\gamma>0$, and let $q\in C^1(I;\ell^\infty(V))$.  Suppose that
\begin{equation}\label{eq:balanced-system}
    \dot q_v(r)
    =-\gamma\bigl(1+\theta_v(r)\bigr)q_v(r)
     -\gamma\sum_{u\sim v}\theta_{vu}(r)q_u(r)
     +g_v(r)
\end{equation}
for $r\in I$, where $g\in C(I;\ell^\infty(V))$ and
\[
    \theta_{vu}(r)\geq0,
    \qquad
    \theta_v(r)=\sum_{u\sim v}\theta_{vu}(r),
    \qquad
   B_*:= \sup_{r\in I,\,v\in V}\theta_v(r)<\infty.
\]
Then
\begin{equation}\label{eq:balanced-bound}
    \norm{q(t)}_\infty
    \leq \e^{-\gamma(t-s)}\norm{q(s)}_\infty
      +\int_s^t\e^{-\gamma(t-r)}\norm{g(r)}_\infty\dd r.
\end{equation}
\end{lemma}

\begin{proof}
For $r\in I$, let $\mathcal L_r$ denote the homogeneous operator in \eqref{eq:balanced-system}, namely,
\[
    (\mathcal L_r z)_v:=-\gamma\bigl(1+\theta_v(r)\bigr)z_v
      -\gamma\sum_{u\sim v}\theta_{vu}(r)z_u.
\]
If $0<\varepsilon<[\gamma(1+B_*)]^{-1}$, then, for every $z\in\ell^\infty(V)$,
\[
    |((\mathrm{Id}+\varepsilon\mathcal L_r)z)_v| \leq \bigl(1-\varepsilon\gamma(1+\theta_v(r))\bigr)|z_v|
       +\varepsilon\gamma\sum_{u\sim v}\theta_{vu}(r)|z_u|\leq (1-\gamma\varepsilon)\norm z_\infty,
\]
where the last step used $\sum_{u\sim v}\theta_{vu}(r)=\theta_v(r)$.  Thus $\norm{\mathrm{Id}+\varepsilon\mathcal L_r}_{\infty\to\infty}\leq1-\gamma\varepsilon$ and 
\begin{equation}\label{eq:lognorm-bound}
    \operatorname{Log}_\infty(\mathcal L_r)\leq-\gamma
\end{equation}
uniformly in $r$.

Since $q$ is continuously differentiable in the supremum norm,
\[
    q(r+\varepsilon)
    =(\mathrm{Id}+\varepsilon\mathcal L_r)q(r)
       +\varepsilon g(r)+o(\varepsilon)
\]
in $\ell^\infty(V)$ as $\varepsilon\to 0+$.  The real-valued function $Q(r):=\norm{q(r)}_\infty$ is Lipschitz on $I$ and hence differentiable almost everywhere.  At its points of differentiability, the operator bound \eqref{eq:lognorm-bound} gives
\[
    Q'(r)\leq-\gamma Q(r)+\norm{g(r)}_\infty.
\]
Multiplying by $\e^{\gamma r}$ and integrating yields
\[
    \e^{\gamma t}Q(t)-\e^{\gamma s}Q(s)
    \leq\int_s^t\e^{\gamma r}\norm{g(r)}_\infty\dd r,
\]
which is \eqref{eq:balanced-bound}.  
\end{proof}

We now compare two bounded positive complete solutions.  After the power transformation, the forcing term on each edge is the sum of two vertex contributions.  We collect their differences into vertex variables.  In the resulting vertex equation, each edge contributes equally to the diagonal coefficient and to the coefficient at its other endpoint.  This gives the row-sum identity needed in Lemma \ref{lem:lognorm-estimate}.

Write
\[
    K_\beta:=\max\{1,\beta\}.
\]

\begin{proposition}\label{prop:delayed-contraction}
Let $x,\widetilde x$ be two bounded positive complete solutions of \eqref{eq:x-ode}, and set $y_e=x_e^\gamma$, $\widetilde y_e=\widetilde x_e^\gamma$. Then for all $s<t$,
\begin{equation}\label{eq:delayed-contraction}
    \norm{y(t)-\widetilde y(t)}_\infty
    \leq
    [1+2K_\beta(t-s)]\e^{-\gamma(t-s)}
       \norm{y(s)-\widetilde y(s)}_\infty.
\end{equation}
\end{proposition}

\begin{proof}
Write $h=y-\widetilde y$. For $e\in E$, define
\begin{equation}\label{eq:secant-a}
    a_e(r):=\int_0^1\beta
       \bigl(\theta y_e(r)+(1-\theta)\widetilde y_e(r)\bigr)^{\beta-1}
       \dd\theta.
\end{equation}
The fundamental theorem of calculus gives $y_e^\beta-\widetilde y_e^\beta=a_eh_e$.  For $v\in e$, define
\begin{equation}\label{eq:c-ve}
    c_{ve}(r):=\frac{p_v a_e(r)}
       {R_v(y(r))R_v(\widetilde y(r))},
    \qquad
    ( \mathcal{C}_r z)_v:=\sum_{e\in E_v}c_{ve}(r)z_e.
\end{equation}
It follows by definition that
\begin{equation}\label{eq:rho-difference}
    \rho_v(y(r))-\rho_v(\widetilde y(r))=-(\mathcal{C}_rh(r))_v.
\end{equation}
 Fix $s<t$ and define 
\begin{equation}\label{eq:q-def}
    q_v(r):=\gamma\int_s^r\e^{-\gamma(r-\tau)}
        (\mathcal{C}_\tau h(\tau))_v\dd\tau, \quad r \in [s,t].
\end{equation}
In particular, $q(s)=0$ and $\dot q+\gamma q=\gamma \mathcal{C}_rh$. Let $\mathsf M$ be the unsigned incidence matrix and define
\[
    (\mathsf Mz)_v:=\sum_{e\in E_v}z_e,
    \qquad
    (\mathsf M^{\mathsf T}w)_{\{u,v\}}:=w_u+w_v.
\]
Here $\mathsf M^{\mathsf T}:\ell^\infty(V)\to\ell^\infty(E)$ simply sums the two endpoint coordinates, so $\norm{\mathsf M^{\mathsf T}}_{\infty\to\infty}\leq2$.  Equations \eqref{eq:y-ode} and \eqref{eq:rho-difference} yield
\begin{equation}\label{eq:h-factorized}
    \dot h=-\gamma\bigl(h+\mathsf M^{\mathsf T}\mathcal{C}_rh\bigr).
\end{equation}
Applying \eqref{eq:variation-of-constants} to \eqref{eq:h-factorized} gives, for $e=\{u,v\}$,
\begin{equation}\label{eq:h-q}
    h_e(r)=\e^{-\gamma(r-s)}h_e(s)-q_u(r)-q_v(r).
\end{equation}
We first derive the equation for $q$ coordinate-wise.  Substituting \eqref{eq:h-q} into $\dot q+\gamma q=\gamma\mathcal C_rh$, we use
\[
    (\mathcal{C}_r\mathsf M^{\mathsf T}q)_v
    =\sum_{u\sim v}c_{v,\{u,v\}}(r)(q_v+q_u)
    =\theta_v(r)q_v+\sum_{u\sim v}\theta_{vu}(r)q_u,
\]
where
\[
    \theta_{vu}(r):=c_{v,\{u,v\}}(r),
    \qquad
    \theta_v(r):=\sum_{u\sim v}\theta_{vu}(r).
\]
Each edge contributes once to $q_v$ and once to the coordinate at its other endpoint.  We obtain
\begin{equation}\label{eq:q-equation}
    \dot q_v(r)=-\gamma(1+\theta_v(r))q_v(r)
       -\gamma\sum_{u\sim v}\theta_{vu}(r)q_u(r)+\gamma\e^{-\gamma(r-s)}(\mathcal{C}_rh(s))_v.
\end{equation}
The diagonal interaction coefficient equals the off-diagonal row sum, as required in Lemma \ref{lem:lognorm-estimate}.

To apply that lemma, we now verify the uniform coefficient bound and the required differentiability in the supremum norm.

First, we show that $\mathcal{C}_r:\ell^\infty(E)\to\ell^\infty(V)$ is bounded uniformly in $r\in[s,t]$.  For $a,b>0$,
\begin{equation}\label{eq:power-secant}
    \frac{a^\beta-b^\beta}{a-b}
    \leq K_\beta\frac{a^\beta+b^\beta}{a+b},
\end{equation}
with the derivative used when $a=b$.  By symmetry, it suffices to consider $a>b$.  Put $u=\frac12\log(a/b)>0$.  Then
\[
    \frac{a^\beta-b^\beta}{a-b}
       \frac{a+b}{a^\beta+b^\beta}
    =\frac{\tanh(\beta u)}{\tanh u}.
\]
For $0<\beta\leq1$, monotonicity gives $\tanh(\beta u)\leq\tanh u$.  For $\beta\geq1$, concavity and $\tanh0=0$ give $\tanh u\geq\beta^{-1}\tanh(\beta u)$.  This proves \eqref{eq:power-secant}.  The case $a=b$ follows by taking a limit. For $v\in e$, Lemma \ref{lem:complete-feedback} gives
\[
    y_e+\widetilde y_e
    \geq\gamma p_v
       \frac{R_v(y)+R_v(\widetilde y)}
            {R_v(y)R_v(\widetilde y)}.
\]
Combining this with \eqref{eq:power-secant} yields
\[
    c_{ve}(r)\leq
    \frac{K_\beta}{\gamma}
    \frac{y_e(r)^\beta+\widetilde y_e(r)^\beta}
         {R_v(y(r))+R_v(\widetilde y(r))}.
\]
Since the coefficients are nonnegative, the operator norm of $\mathcal{C}_r$ equals the supremum of its row sums.  Summing over $E_v$ therefore gives
\begin{equation}\label{eq:C-bound}
    \norm{\mathcal{C}_r}_{\infty\to\infty}
    =\sup_v\sum_{e\in E_v}c_{ve}(r)
    \leq\frac{K_\beta}{\gamma}.
\end{equation}
Only the combined coefficients $c_{ve}$ are bounded here.  Neither $a_e$ nor $p_v/[R_v(y)R_v(\widetilde y)]$ needs to be uniformly bounded.

We also need continuity in operator norm.  Let $L=\max\{\gamma,\alpha\}$ as in the proof of Lemma \ref{lem:complete-feedback}.  Fix an edge $e$ and put $z_\theta=\theta y_e+(1-\theta)\widetilde y_e$.  By \eqref{eq:log-y-bound}, $|\dot z_\theta|\leq Lz_\theta$.  On each compact time interval, $z_\theta$ has a positive lower bound uniform in $\theta\in[0,1]$.  We may therefore differentiate \eqref{eq:secant-a} under the integral to get
\[
    |\dot a_e|
    \leq\int_0^1\beta|\beta-1|z_\theta^{\beta-2}
             |\dot z_\theta|\dd\theta
    \leq|\beta-1|L a_e.
\]
 Since all factors in $c_{ve}$ are positive,
\[
    \frac{\dot c_{ve}}{c_{ve}}
    =\frac{\dot a_e}{a_e}
      -\frac{\dd}{\dd r}\log R_v(y(r))
      -\frac{\dd}{\dd r}\log R_v(\widetilde y(r)).
\]
The derivative bound \eqref{eq:log-R-bound} for $R_v$ now implies that for any edge $e\in E$ and $v \in e$,
\[
    |\dot c_{ve}|\leq Hc_{ve},
    \qquad \text{where } H:=(|\beta-1|+2\beta)L.
\]
For $s<t$, integration and \eqref{eq:C-bound} give
\[
\norm{\mathcal{C}_t-\mathcal{C}_s}_{\infty\to\infty}\leq\sup_v\sum_{e\in E_v}\int_s^t|\dot c_{ve}(r)|\dd r\leq\frac{HK_\beta}{\gamma}(t-s).
\]
Combined with Lemma \ref{lem:complete-feedback}, this shows that the map $r\mapsto \mathcal{C}_rh(r)$ is continuous in $\ell^\infty(V)$. In view of \eqref{eq:q-def}, we see that $q$ is norm-continuous.  The criterion \eqref{eq:sup-diff}, applied to $\dot q=\gamma(\mathcal{C}_rh-q)$, then gives $q\in C^1([s,t];\ell^\infty(V))$.

By \eqref{eq:C-bound}, the coefficients in \eqref{eq:q-equation} satisfy $\theta_v(r)\leq K_\beta/\gamma$, and the last term has norm at most
\[
    K_\beta\e^{-\gamma(r-s)}\norm{h(s)}_\infty.
\]
Since $q(s)=0$, Lemma \ref{lem:lognorm-estimate} yields
\[
    \norm{q(t)}_\infty \leq K_\beta\norm{h(s)}_\infty
       \int_s^t\e^{-\gamma(t-r)}\e^{-\gamma(r-s)}\dd r=K_\beta(t-s)\e^{-\gamma(t-s)}\norm{h(s)}_\infty,
\]
which proves \eqref{eq:delayed-contraction} in view of \eqref{eq:h-q}.
\end{proof}

\begin{remark}\label{rem:complete-only}
The uniform coefficient bound \eqref{eq:C-bound} uses completeness through Lemma \ref{lem:complete-feedback}.  Proposition \ref{prop:delayed-contraction} therefore compares bounded positive complete solutions.  It does not give a uniform convergence rate for arbitrary forward solutions or for the urn process.
\end{remark}

We now prove Theorem \ref{thm:liouville}.

\begin{proof}[Proof of Theorem \ref{thm:liouville}]
The existence follows from Proposition \ref{prop:equilibrium-existence}.  Fix such an equilibrium $\mu$ and put $y_e^*=\mu_e^\gamma$.  Let $x$ be a bounded positive complete solution and set $y_e=x_e^\gamma$.  Proposition \ref{prop:equilibrium-existence} and Lemma \ref{lem:complete-feedback} bound both $y$ and $y^*$ from above by $(2 p_{*})^\gamma$.  Fix time $t\in\R$.  Since $x$ is complete, we may compare the two trajectories on $[t-T,t]$ for every $T>0$.  Proposition \ref{prop:delayed-contraction} gives
\[
 \norm{y(t)-y^*}_\infty
    \leq(1+2K_\beta T)\e^{-\gamma T}
        \norm{y(t-T)-y^*}_\infty\leq(2 p_{*})^\gamma(1+2K_\beta T)\e^{-\gamma T}
      \longrightarrow 0,
\]
as $T\to\infty$.  Hence $y(t)=y^*$.  Since $t$ is arbitrary, one has $x\equiv\mu$.  Every other non-vanishing equilibrium is bounded by Proposition \ref{prop:equilibrium-existence}.  Applying the same argument to its constant trajectory proves the uniqueness.
\end{proof}

\section{Logarithmic-time stochastic approximation}\label{sec:sa}

We prove Theorem \ref{thm:main} in this section. We first establish local upper and lower bounds for the normalized process.  We then extract complete trajectories from its logarithmic-time shifts and apply Theorem \ref{thm:liouville}.

\subsection{Local bounds and persistence}

Write $Q_v(t)$ for the number of rings of the clock at $v$ by time $t$.  For $e=\{u,v\}$,
\[
    N_t(e)\leq N_0(e)+Q_u(t)+Q_v(t).
\]
The strong law of large numbers for the Poisson process gives
\begin{equation}\label{eq:upper-limsup}
    \limsup_{t\to\infty}X_t(e)\leq b_e\leq2 p_{*}
    \qquad\text{almost surely}.
\end{equation}

For unit firing rates and unit initial weights, persistence follows from \cite[Lemma~3.6]{MR4275961}.  We give a local comparison proof that also covers the present assumptions.

\begin{proposition}[Persistence]\label{prop:persistence}
For every fixed edge $e\in E$,
\begin{equation}\label{eq:persistence}
    m_e\leq\liminf_{t\to\infty}X_t(e)
    \leq\limsup_{t\to\infty}X_t(e)\leq b_e
    \qquad\text{almost surely}.
\end{equation}
\end{proposition}

\begin{proof}
Fix $e=\{u,v\}$ and $B>2p_*$.  The upper bound follows from \eqref{eq:upper-limsup}.  The set $E_u\cup E_v$ is finite. For $t\geq 1$, let $\mathcal E_t$ be the event that
\begin{equation}\label{eq:local-upper-event}
    N_t(f)\leq Bt
    \qquad \text{for all } f\in E_u\cup E_v.
\end{equation}
Then, by \eqref{eq:upper-limsup}, almost surely, $\mathcal E_t$ holds for all sufficiently large $t$.  For each integer $T\geq1$, let 
\[
\tau_T:=\inf\{t\geq T: \mathcal E_t \text{ does not hold.}\}
\]
be the first time $t\geq T$ at which one of the inequalities in \eqref{eq:local-upper-event} fails, with the convention that $\inf\varnothing=\infty$.  For $T\leq t<\tau_T$, the jump intensity of $N(e)$ satisfies
\begin{equation}\label{eq:birth-lower}
    \lambda_e(t):=\Phi_e(N_{t-})
    \geq\kappa_e t^{-\alpha}N_{t-}(e)^\alpha,
    \qquad
    \kappa_e:=B^{-\alpha}\eta_e>0.
\end{equation}

We couple $N(e)$ with pure-birth processes $Y^{(T)}$ for integers $T\geq 1$. We first realize $N(e)$ by a standard Poisson thinning construction.  Superposing the two independent Poisson clocks at $u$ and $v$, of rates $p_u$ and $p_v$, yields a Poisson clock of rate $b_e=p_u+p_v$.  Conditional on a ring at time $s$, the firing vertex $w\in\{u,v\}$ is selected with probability $p_w/b_e$, and, given $w$, the edge $f\in E_w$ is selected with probability $N_{s-}(f)^\alpha/S_w(N_{s-})$.  Equivalently, both choices are recovered from a single uniform mark on $[0,b_e]$ by partitioning $[0,b_e]$ into intervals of lengths $p_w N_{s-}(f)^\alpha/S_w(N_{s-})$ indexed by pairs $(w,f)$, in a fixed order with the two intervals for $(u,e)$ and $(v,e)$ placed first, and the interval containing the mark identifies $(w,f)$.  This re-ordering depends only on the past $\mathcal F_{s-}$ and preserves the conditional law of the firing vertex and the selected edge.  Together with the ring times, these marks form a Poisson random measure on $[0,\infty)\times[0,b_e]$ with intensity $\dd s\,\dd x$, see e.g. \cite[Theorem 5.6]{MR3791470}.  The first two intervals have total length $\lambda_e(s)$, so $N(e)$ jumps precisely at the points $(s,x)$ with $x\leq\lambda_e(s)$.  On an extension of the probability space, add an independent Poisson random measure on $[0,\infty)\times(b_e,\infty)$ with the same intensity, and let $\Lambda$ denote the combined measure on $[0,\infty)\times[0,\infty)$. In particular, $\Lambda$ is a Poisson random measure on $[0,\infty)\times[0,\infty)$ with intensity $\dd s\,\dd x$.

Put $r(s,z):=\kappa_e s^{-\alpha}z^\alpha$.  For every integer $T\geq1$, start $Y^{(T)}$ from $Y_T^{(T)}=N_0(e)$ and let it jump by one at the points $(s,x)\in\Lambda$ with $s>T$ and $x\leq r(s,Y_{s-}^{(T)})$.  Its jump rate is $r$.  Since $\alpha<1$, for $s\geq T$ this rate is bounded by $C_T(1+z)$ for some finite constant $C_T$, so $Y^{(T)}$ is non-explosive and defined for every $t\geq T$, including after $\tau_T$.

A jump of $Y^{(T)}$ at time $s$ occurs at an atom $(s,x)\in\Lambda$ with $x\leq r(s,Y_{s-}^{(T)})$.  On the event $\{s<\tau_T,\ Y_{s-}^{(T)}\leq N_{s-}(e)\}$, the monotonicity of $r(s,\cdot)$ and the inequality \eqref{eq:birth-lower} yield
$$
    x\leq r(s,Y_{s-}^{(T)})\leq r(s,N_{s-}(e))\leq\lambda_e(s)\leq b_e,
$$
so $(s,x)$ is also an atom of $N(e)$.  The two processes have only finitely many jumps on compact time intervals. Starting from $Y_T^{(T)}\leq N_T(e)$ and comparing at their successive jump times therefore gives
\begin{equation}\label{eq:birth-comparison}
    Y_t^{(T)}\leq N_t(e),\qquad T\leq t<\tau_T.
\end{equation}
Note that all these comparison processes $Y^{(T)}$ are built from the same $\Lambda$.

We show that almost surely,
\begin{equation}
    \label{eq:birth-asymptotic}
    \lim_{t\to \infty} \frac{Y_t^{(T)}}t = \kappa_e^{1/\gamma}.
\end{equation}
Let $(\widetilde Y_s)_{s\geq0}$ be a pure-birth process started from $\widetilde Y_0=N_0(e)$, with jump rate $(N_0(e)+k)^\alpha$ at state $N_0(e)+k$.  The deterministic time change
\[
    A_T(t):=\kappa_e\int_T^t s^{-\alpha}\dd s
       =\frac{\kappa_e}{\gamma}(t^\gamma-T^\gamma)
\]
shows that $(\widetilde Y_{A_T(t)})_{t\geq T}$ has the same law as $Y^{(T)}$.  It therefore suffices to prove
\begin{equation}\label{eq:Ytildeasy}
    \lim_{a\to\infty}\frac{\widetilde Y_a}{a^{1/\gamma}}
    =\gamma^{1/\gamma}
    \qquad\text{almost surely}.
\end{equation}
The process $\widetilde Y$ is non-explosive and its successive jump times are 
\[
    H_n=\sum_{k=0}^{n-1}\frac{\xi_k}{(N_0(e)+k)^\alpha}, \quad n\geq 1,
\]
where $(\xi_k)_{k\geq0}$ are independent mean-one exponential variables.
Write $\bar{H}_n=\E H_n$. By a comparison with the integral $\int u^{-\alpha} \dd u$, we see that $\bar{H}_n\sim n^\gamma/\gamma$, and thus,
\[
    \sum_{k=0}^\infty
       \frac{1}{(N_0(e)+k)^{2\alpha}\bar{H}_{k+1}^2}<\infty.
\]
The independent centered series
\[
    \sum_{k=0}^\infty
       \frac{\xi_k-1}{(N_0(e)+k)^\alpha \bar{H}_{k+1}}
\]
has summable variances, so Kolmogorov's convergence theorem shows that it converges almost surely.  Kronecker's lemma then yields $(H_n-\bar{H}_n)/\bar{H}_n\to0$.  Thus $H_n\sim n^\gamma/\gamma$ almost surely. For $s\geq 0$, define $J(s):=\max\{n\geq0:H_n\leq s\}$, with $H_0=0$.  Since $H_{J(s)}\leq s<H_{J(s)+1}$, inversion of this asymptotic gives $J(s)\sim(\gamma s)^{1/\gamma}$ almost surely.  The identity $\widetilde Y_s=N_0(e)+J(s)$ proves \eqref{eq:Ytildeasy}, and the time change proves \eqref{eq:birth-asymptotic}.

For each integer $T$, \eqref{eq:birth-asymptotic} holds almost surely.  By taking the countable intersection of these events, we may use the estimate simultaneously for all such $T$.  Almost surely, \eqref{eq:local-upper-event} holds for all $t\geq T$ for some integer $T$, so $\tau_T=\infty$.  The comparison then gives
\[
    \liminf_{t\to\infty}X_t(e)
    \geq\left(\frac{\eta_e}{B^\alpha}\right)^{1/\gamma}
    \qquad\text{almost surely}.
\]
Apply this conclusion for $B=2p_*+1/n$, $n\geq1$, on their countable intersection of probability-one events, and let $n\to\infty$.  This gives the lower bound $m_e$ in \eqref{eq:persistence}.
\end{proof}

\subsection{Logarithmic time and complete limit}

We use the natural filtration generated by the clocks and the choices made at their rings.  The predictable intensity of $N_t(e)$ is $\Phi_e(N_{t-})$, where $\Phi$ is defined in \eqref{eq:S-Phi-F}.  Since $N_{s-}=N_s$ for Lebesgue-almost every $s$, we may write the compensator without left limits.  Consequently,
\begin{equation}\label{eq:M-def}
    M_t(e):=N_t(e)-N_0(e)-\int_0^t\Phi_e(N_s)\dd s
\end{equation}
is a square-integrable martingale.  Since $0\leq\Phi_e\leq2 p_{*}$,
\begin{equation}\label{eq:M-qv}
    \langle M(e)\rangle_t
    =\int_0^t\Phi_e(N_s)\dd s
    \leq2 p_{*}t.
\end{equation}
Define, for $t\geq1$,
\begin{equation}\label{eq:Z-def}
    Z_t(e):=\int_1^t\frac1s\dd M_s(e).
\end{equation}
Then
\begin{equation}\label{eq:Z-qv}
    \E\langle Z(e)\rangle_\infty
    \leq2 p_{*}\int_1^\infty\frac{\dd s}{s^2}
    <\infty.
\end{equation}
Thus $Z_t(e)$ converges almost surely and in $L^2$ as $t\to\infty$.

We put the process on logarithmic time by setting
\begin{equation}\label{eq:Xhat-def}
    \widehat X_s(e):=X_{\e^s}(e)
    =\e^{-s}N_{\e^s}(e),
    \qquad s\geq0.
\end{equation}
Since $1/t$ is continuous and of bounded variation,
$\dd X_t(e)=t^{-1}\dd N_t(e)-t^{-1}X_t(e)\dd t$.
Integrating over $[1,t]$ and using \eqref{eq:M-def} gives
\[
    X_t(e)-X_1(e)
    =\int_1^t\frac{F_e(X_s)}s\dd s
      +\int_1^t\frac{\dd M_s(e)}s,
\]
where the homogeneity of degree zero of $\Phi$ was used in the form $\Phi_e(N_s)=\Phi_e(X_s)$.  The change of variables $s=\e^r$ now gives, for $0\leq a\leq b$,
\begin{equation}\label{eq:log-SA}
    \widehat X_b(e)-\widehat X_a(e)
    =\int_a^bF_e(\widehat X_r)\dd r
      +Z_{\e^b}(e)-Z_{\e^a}(e).
\end{equation}

Because $E$ is countable, \eqref{eq:persistence}, \eqref{eq:upper-limsup}, and the convergence of \eqref{eq:Z-def} hold simultaneously for every edge on a single event of probability one.  From now on, we work on this event.

\begin{proposition}
\label{prop:complete-orbit-compactness}
On the probability-one event fixed above, for any sequence $(s_n)$ tending to infinity, there is a subsequence $(s_{n_k})$ and a positive complete solution $x:\R\to(0,\infty)^E$ of \eqref{eq:x-ode} such that
\begin{equation}\label{eq:path-convergence}
    \lim_{k\to\infty}\sup_{ |r|\leq T}
    |\widehat X_{s_{n_k}+r}(e)-x_e(r)|
    =0
\end{equation}
for every $0<T<\infty$ and every $e\in E$.  Moreover,
\begin{equation}\label{eq:limit-upper}
    m_e\leq x_e(r)\leq b_e\leq2p_*
    \qquad\text{for every }e\in E\text{ and }r\in\R.
\end{equation}
\end{proposition}

\begin{proof}
Fix an edge $e$ and $T>0$.  By \eqref{eq:persistence}, $X_t(e)\geq m_e/2$ for all sufficiently large $t$.  Together with \eqref{eq:upper-limsup}, this implies that, for all sufficiently large $n$ (with the threshold allowed to depend on $e$ and $T$),
\begin{equation}\label{eq:window-box}
    m_e/2\leq \widehat X_{s_n+r}(e)\leq2 p_{*}+1
    \qquad\text{for every }r\in[-T,T].
\end{equation}
Since $0\leq\Phi_e\leq2 p_{*}$, it follows that
\[
    |F_e(\widehat X_{s_n+r})|\leq2 p_{*}+1
\]
on the same window.  The almost sure convergence of $Z_t(e)$ also implies
\begin{equation}\label{eq:noise-window}
    \sup_{r,r'\in[-T,T]}
    |Z_{\e^{s_n+r}}(e)-Z_{\e^{s_n+r'}}(e)|
    \longrightarrow0.
\end{equation}
The shifted processes are c\`adl\`ag rather than continuous, so we apply Arzel\`a--Ascoli to a continuous approximation.  For sufficiently large $n$ (such that $s_n\geq T$) and $r\in[-T,T]$, define
\begin{equation}\label{eq:continuous-approximation}
    \bar{Y}_n^e(r)
    :=\widehat X_{s_n-T}(e)
      +\int_{-T}^r F_e(\widehat X_{s_n+\tau})\dd\tau.
\end{equation}
The functions $\bar{Y}_n^e$ are continuous and $(2 p_{*}+1)$-Lipschitz on $[-T,T]$.  By \eqref{eq:log-SA},
\begin{equation}\label{eq:approximation-error}
    \sup_{r\in[-T,T]}
    |\widehat X_{s_n+r}(e)-\bar{Y}_n^e(r)|
    \leq
    \sup_{r\in[-T,T]}
    |Z_{\e^{s_n+r}}(e)-Z_{\e^{s_n-T}}(e)|
    \longrightarrow0.
\end{equation}
Thus the shifted $e$-coordinates are uniformly approximated by a bounded equi-Lipschitz family on every compact interval.

To apply the Arzel\`a--Ascoli theorem simultaneously to all coordinates and compact time intervals, enumerate $E=\{e_j:1\leq j\leq |E|\}$, allowing $|E|=\infty$.  We construct the desired subsequence inductively.  Set $s_n^{(0)}:=s_n$ for every $n$.  At stage $m\geq1$, given the previous subsequence $(s_n^{(m-1)})$, let $\bar{Y}_n^{e_j}$ denote the approximation \eqref{eq:continuous-approximation} built from $s_n^{(m-1)}$ in place of $s_n$, with $T=m$.  These functions are uniformly bounded, and equi-continuous on $[-m,m]$.  Arzel\`a--Ascoli then yields a further subsequence $(s_n^{(m)})$ along which $\{\bar{Y}^{e_j}\}_{1\leq j \leq \min\{m,|E|\}}$ converge uniformly on $[-m,m]$ to continuous limits.  By \eqref{eq:approximation-error}, the shifted processes $r\mapsto\widehat X_{s_n^{(m)}+r}(e_j)$, $1\leq j \leq \min\{m,|E|\}$, converge to the same limits.  The diagonal subsequence $s_{n_k}:=s_{n_k}^{(k)}$ then has the following property: for every edge $e_j$ and every $m\geq j$, the functions $(r\mapsto\widehat X_{s_{n_k}+r}(e_j))_{k\geq1}$ converge uniformly on $[-m,m]$.  Writing $x_{e_j}(r)$ for the limit, the limits on overlapping windows $[-m,m]$ and $[-m',m']$ agree on their intersection, so each $x_e:\R\to(0,\infty)$ is well defined and satisfies \eqref{eq:path-convergence}.  Positivity follows from the lower bound in \eqref{eq:window-box}.  The liminf and limsup bounds in \eqref{eq:persistence} give $m_e\leq x_e(r)\leq b_e$ for every $e$ and $r$.

It remains to identify the limiting equation.  Fix $e=\{u,v\}$ and a compact time interval $I$.  The coordinate $F_e$ depends only on the finitely many coordinates in $E_u\cup E_v$.  The corresponding limiting coordinates are positive and continuous on $I$.  We may therefore choose $\delta>0$ such that $x_f(r)\geq2\delta$ for every $f\in E_u\cup E_v$ and $r\in I$.  By \eqref{eq:path-convergence}, these coordinates of the shifted paths all lie in $[\delta,2 p_{*}+1]$ for sufficiently large $k$.  The function $F_e$ is continuous on this finite-dimensional compact box, where both its denominators are positive.  Its uniform continuity gives
\[
    F_e(\widehat X_{s_{n_k}+\cdot})
    \longrightarrow F_e(x(\cdot))
\]
uniformly on $I$.  At the same time, the martingale increment in \eqref{eq:log-SA} tends to zero uniformly on $I$ by \eqref{eq:noise-window}.  Passing to the limit in \eqref{eq:log-SA} yields
\[
    x_e(b)-x_e(a)=\int_a^bF_e(x(r))\dd r
\]
for all $a<b$.  Since $F_e(x(\cdot))$ is continuous, $x_e$ is continuously differentiable and satisfies \eqref{eq:x-ode}.  The construction works on every compact interval, so $x$ is a complete solution.
\end{proof}

The extraction above uses convergence of each coordinate.  The random vectors $\widehat X_s$ need not belong to $\ell^\infty(E)$, since the initial weights may be unbounded over $E$.  The limiting complete solution, however, satisfies the uniform bound \eqref{eq:limit-upper}.  We can therefore apply the comparison in $\ell^\infty(E)$ from Section \ref{sec:ode} to this limit.

\begin{proof}[Proof of Theorem \ref{thm:main}]
The existence and uniqueness of the non-vanishing equilibrium follow from Proposition \ref{prop:equilibrium-existence} and Theorem \ref{thm:liouville}.  It remains to prove convergence of the random process.

Suppose, on the probability-one event fixed above, that \eqref{eq:main-convergence} fails for some edge $e$.  Then there are $\varepsilon>0$ and a sequence $s_n\to\infty$ such that
\begin{equation}\label{eq:contradiction-sequence}
    |\widehat X_{s_n}(e)-\mu_e|\geq\varepsilon
    \qquad\text{for every }n.
\end{equation}
By Proposition \ref{prop:complete-orbit-compactness}, a subsequence of the shifted processes converges as in \eqref{eq:path-convergence} to a positive complete solution $x$.  We retain the notation $(s_n)$ for this subsequence.  The limit satisfies $0<x_f(r)\leq2 p_{*}$ for all $f$ and $r$.  Theorem \ref{thm:liouville} gives $x(r)\equiv\mu$.  In particular,
\[
    \widehat X_{s_n}(e)\longrightarrow x_e(0)=\mu_e,
\]
contradicting \eqref{eq:contradiction-sequence}.  Thus $X_t(e)\to\mu_e$ for every edge $e$.  Since all the preceding statements hold simultaneously on a single probability-one event, the convergence is simultaneous over the countable edge set.  
\end{proof}

\section{Acknowledgments}

The author would like to thank Angel, Hermon and Yao for sharing an earlier version of their manuscript and for bringing the problem to his attention. During the preparation of this manuscript, the author used GPT 6 to improve language and assist with proofreading. The author is supported by the China Postdoctoral Science Foundation under Grant Numbers 2025M773086 and 2026T190815, and the National Natural Science Foundation of China under Grant Number 12271284.

\bibliographystyle{plain}
\bibliography{warm_references}

\end{document}